\documentclass{amsart}
\usepackage{amsthm,amstext,amsmath,amscd,amssymb,latexsym}
\usepackage{color}
\usepackage[matrix,arrow]{xy}
\usepackage{amsfonts,enumerate,array}
\usepackage[colorlinks=true]{hyperref}
\usepackage{todonotes}
	\usepackage{verbatim}
\usepackage[toc,page]{appendix}

\newcommand{\PP}{{\mathbb P}}
\newcommand{\pp}{{\mathbb P}}
\newcommand{\FF}{{\mathbb F}}
\newcommand{\F}{{\mathbb F}}
\newcommand{\ZZ}{{\mathbb Z}}
\newcommand{\NN}{{\mathbb N}}
\newcommand{\TT}{{\mathbb T}}

\newcommand{\LL}{{\mathcal L}}

\newcommand{\ls}{{\mathcal L}}
\newcommand{\lsk}{\hat{{\mathcal L}}}

\newcommand{\hh}{{\rm{h}}}

\DeclareMathOperator{\vdim}{vdim}
\DeclareMathOperator{\edim}{edim}
\DeclareMathOperator{\ldim}{ldim}

\DeclareMathOperator{\Pic}{Pic}

\newcommand{\Mac}{{\texttt {Macaulay2}}}

\newcommand{\paper}{: \begin{it}}
\newcommand{\jour }{, \end{it}}

\newtheorem{theorem}{Theorem}[section]
\newtheorem{lemma}[theorem]{Lemma}
\newtheorem{proposition}[theorem]{Proposition}
\newtheorem{corollary}[theorem]{Corollary}
\newtheorem{conjecture}[theorem]{Conjecture}

\theoremstyle{definition}

\theoremstyle{remark}
\newtheorem{remark}[theorem]{Remark}

\numberwithin{equation}{section}

\begin{document}

\title{On linear systems of $\PP^3$ with nine base points}

\author{Maria Chiara Brambilla}
\email{{\tt brambilla@dipmat.univpm.it}}
\address{Universit\`a Politecnica delle Marche, 
via Brecce Bianche, I-60131 Ancona, Italy}

\author{Olivia Dumitrescu}
\email{{\tt  dumitrescu@math.uni-hannover.de}}
\address{Institut f\"ur Algebraische Geometrie GRK 1463, Welfengarten 1, 30167 Hannover, Germany}

\author{Elisa Postinghel}
\email{{\tt elisa.postinghel@wis.kuleuven.be}}
\address{KU Leuven, Department of Mathematics, Celestijnenlaan 200B, 3001 Heverlee,
Belgium}

\thanks{The first author is partially supported by MIUR and INDAM.
The second author is a member of 
the Simion Stoilow Institute of Mathematics of the 
Romanian Academy. The third author is supported by the Research Foundation - Flanders (FWO)}

\keywords{Fat points, degeneration techniques, Laface-Ugaglia Conjecture, base locus, quadric surface}

\subjclass[2010]{Primary: 14C20. Secondary: 14J70, 14J26}

\begin{abstract}
We study special linear systems of surfaces of $\PP^3$ interpolating nine points in general position
 having a quadric as fixed component.
By performing degenerations in
 the blown-up space, we interpret the quadric obstruction in terms of linear obstructions for a quasi-homogeneous class.
By degeneration we also prove a Nagata type result for the blown-up projective plane in points that implies a base locus lemma for the quadric. 
As an application we establish Laface-Ugaglia Conjecture for linear systems with multiplicities bounded by $8$
and for homogeneous linear systems  with multiplicity $m$ and degree up to  $2m+1$.
\end{abstract}

\maketitle
\section{Introduction}

The theory of linear systems is a classical object of study which is related 
to secant varieties, polynomial interpolation and to several interesting 
recently discovered applications.
Even if linear systems have been studied for more than a century, basic questions, 
such as the dimensionality problem, are still open in general.

We denote by $ \LL= \LL_{n,d} (m_1,\ldots,m_s)$ the linear system of hypersurfaces 
of degree $d$ in $\PP^n$ interpolating $s$ points in general 
position with multiplicities respectively $m_1,\ldots,m_s$.
A linear system is said to be {\it non-special} if it has the (affine) expected dimension, 
which is $\edim(\LL) = \max(\vdim(\LL), 0)$, where  the (affine) 
virtual dimension $\vdim(\LL)$ is defined as
$$
\vdim(\LL) =\binom{n+d}{n}-\sum_{i=1}^s\binom{n+m_i-1}{n}.
$$

{\it Special linear systems} are those that have dimension strictly higher than the 
expected one
and the {\it speciality} of the system is  the difference 
$$\dim(\LL)-\edim(\LL)=\hh^1(\LL)\ge0.$$

In general, computing the dimension of the linear systems is a challenging task. 
In order to classify the special linear systems, one has to understand first what
are the \emph{obstructions}, namely what are the varieties 
 that, whenever contained with
multiplicity in the base locus of $\ls$, generate speciality. 
In \cite{bocci1,bocci2} these obstructions are named {\it special effect varieties}.

The well-known Alexander-Hirschowitz Theorem (\cite{AlHi}, see also \cite{ale-hirsch, Po}), which concerns
the case of linear systems with double points in $\PP^n$, provides a list of special systems where the 
special effect varieties are linear cycles (when $d=2$), 
a rational normal curve (when $d=3$)
or a quadric hypersurface (when $d=4$).

For higher multiplicities, the planer case has been deeply investigated by many authors.
For $n=2$, the famous Segre-Harbourne-Gimigliano-Hirschowitz Conjecture states that 
the obstructions are given by $(-1)$-curves
\cite{Gimigliano,  Harbourne3, Hir2, Segre} (see also \cite{Ciliberto,  CM1, CM2}).

For the higher dimensional case, in \cite{BDP,  DP} the authors extensively studied the linear special effect varieties in $\PP^n$.
In particular, knowing the exact contribution to the speciality of any multiple linear cycle 
contained in the base locus allows to introduce 
the notion of {\it linear expected dimension}, (see \cite[Definition 3.2]{BDP}).
We say that a system is {\it linearly non-special} whenever its dimension is equal to
 the linear expected dimension. 
This happens exactly when the only special effect varieties are linear cycles.

The authors devote the paper \cite{n+3} 
to the investigation of linear systems in $\PP^n$ with $n+3$ base points having non linear obstructions.
More precisely, the rational normal 
curve through the points is a special effect curve.

It is a well-known fact that  Cremona reduced 
linear systems of $\PP^3$ do not contain
rational normal curves in their base locus. 
Laface and Ugaglia conjectured 
\cite{laface-ugaglia-TAMS} that for a Cremona reduced 
linear system the only special effect varieties are lines and quadric surfaces determined by nine points.
 The conjecture of Laface-Ugaglia is known to be true if the number of points
 is less or equal than eight \cite{volder-laface}, and when the maximal
 multiplicity of the points is five \cite{ballico-brambilla-caruso-sala}.

In this paper we study
linear systems in $\PP^3$ with at least nine fat points in general position for which
the quadric hypersurface through nine of the base points, 
namely the fixed surface $Q:=\ls_{3,2}(1^9)$,
 is a special effect variety.

The first step is to prove a {\it base locus lemma} for quadric surfaces. 
Even a weak base locus lemma is not obvious to obtain. 
In fact, such results can be obtained as a consequence of {\it Nagata type} results, 
i.e.\ theorems which prove emptiness,
for linear systems in $\PP^2$ with ten points.
In Section \ref{section degree 2m+1} we establish 
a base locus lemma for the quadric surface through nine points 
(see Theorem \ref{quadric base locus})
for a particular class of linear systems in $\PP^3$. 
In order to prove this result, we study the emptiness of
 linear systems with ten points in $\PP^2$, via a suitable degeneration
 technique inspired by \cite{CDMR, CM3}.

The next step is to classify the special linear systems whose special effect varieties are quadrics.
In particular, we focus on the case of (Cremona reduced) linear systems with nine points in $\PP^3$, 
which is the first case where the speciality is not due only to linear obstructions. 

Our goal is to understand precisely how much the quadric surface in the base locus 
contributes to the speciality of the system.
Differently than in the linear case, to give a formula which computes exactly the 
contribution to the speciality seems difficult in general (see Remark \ref{q-in-general} for more details). 

Hence we focus first on some particular classes of linear systems, that are the homogeneous and the
quasi-homogeneous ones.

The first case we study is given by the quasi-homogeneous 
linear systems $\LL_{3,2m}(m^8,a)$, 
for $1\le a\le m$. 
This class of systems behaves surprisingly well, indeed we are able to find an 
easy formula which relates the speciality 
with the multiplicity of the quadric in the base locus, see 
Theorem \ref{theorem quasi-homogeneous}.
The proof of this result is based on a degeneration argument,
 which allows to reduce the ``mysterious'' contribution
of the quadric to the sum of two contributions given by linear 
special effect varieties in the degenerated systems.

We recall that in literature various degeneration arguments have been used to prove 
non-speciality results 
of linear systems in the plane \cite{CM1, CM2, Dum} and 
in higher dimension \cite{LaPo,Po}.

In the case of degree $2m+1$, the relation between the speciality and the quadric becomes 
less clear even in the homogeneous case. 
However in Theorem \ref{degree 2m+1} we classify all the special homogeneous linear 
systems with nine points of multiplicity $m$ and degree $2m+1$.
In order to prove this result we apply the emptiness results mentioned above and proved 
in Section
 \ref{section emptiness}. 

In the last section, as an application of Theorems \ref{theorem quasi-homogeneous} and 
\ref{degree 2m+1}, we show that the Laface-Ugaglia Conjecture holds for linear systems 
of any degree and nine points of multiplicity at most $8$. 
In order to complete this proof as well as  the proof of Theorem \ref{degree 2m+1}, 
some of the computations are made by means of the computer algebra system
 \Mac  \ \cite{macaulay}. 

We want to point out finally that the quadric hypersurfaces are {\it sporadic} special effect varieties. 
Indeed, it is expected (see e.g. the Fr\"oberg-Iarrobino Conjecture for homogeneous linear systems, \cite[Conjecture 4.8]{Chandler}) that they give contribution to the speciality of a linear system only in $\PP^3$ and $\PP^4$.
We think that the understanding  of the case of linear systems in $\PP^3$ with nine points is the initial step 
in order to investigate the special systems obstructed by a quadric.

This article is organized as follows. In Section \ref{preliminaries} 
we give a brief description of
the tools that we will use to prove our results.

In Section \ref{section-quasi-homogeneous} we classify the case $\LL_{3,2m}(m^8,a)$
 and we
give a geometric interpretation, via degenerations, 
of the quadric as special effect surface. 

In Section \ref{section degree 2m+1} we completely classify the case  $\LL_{3,2m+1}(m^9)$; the main results
are  Theorem \ref{quadric base locus} and Theorem \ref{theorem empty}.

In Section \ref{proof LU}, we prove that Laface-Ugaglia Conjecture holds for linear systems with $9$ base
points of multiplicities $m_i\le 8$.

\subsection*{Acknowledgements} 
The authors would like to thank the referee for his/her many useful comments.

\section{Preliminaries}\label{preliminaries}

In this preliminary section we collect general results and techniques 
that will be used throughout the paper.
We point out that by dimension of a linear system $\LL$ we mean the affine dimension $\dim(\LL)=\hh^0(\LL)$, and not the projective dimension.

\subsection{Degenerations}\label{degeneration technicalities}
A natural approach to the dimensionality problem of linear systems is via degenerations. 
Degenerations allow to move the multiple base points of a linear system 
in special position and compute the dimension via a semi-continuity argument. 

In \cite{CM1,CM2} Ciliberto and Miranda exploited a degeneration of the plane, 
originally proposed by Ran \cite{R} to study higher multiplicity 
interpolation problems for planar linear systems with general multiple base points. This approach consists in degenerating the plane to a reducible surface, 
with two components intersecting along a line, and simultaneously degenerating the linear system  to a limit linear system  which is somewhat easier than the original one. 
In particular this degeneration argument allows to use induction either on the degree 
or on the number of imposed multiple points. 
This method was generalized by the third author 
to the higher dimensional cases of $\PP^n$  \cite{Po} and of $({\PP^1})^n$ with Laface \cite{LaPo}.

Let $X\subseteq \pp^N$ be a variety, let
  $\Delta$ be a complex disc with center at the origin and let $\mathcal{X} \rightarrow \Delta$ be a $1$-dimensional
embedded degeneration of $X$ to the union of two varieties $X^1,X^2$, i.e. a $1$-parameter family $\{X_t\}_{t\in\Delta}$ such that  $X_t\cong X,\  t\neq 0$, and  $X_0=X^1\cup X^2$.
Let $\ls_t:=\ls$ be a line bundle on the general fibre.

A limit $\ls_0$ of $\ls_t$ is a  line bundle on $X_0$ obtained as fibred product of a line bundle
 $\ls^1$ on $X^1$ and a line bundle $\ls^2$ on $X^2$
  over the intersection of the restricted line bundles
 ${\ls^1}_{|Y}$ and ${\ls^2}_{|Y}$.
This provides a recursive
 formula for the  dimension of $\ls_0$ in terms of the dimensions of the
 involved linear
 systems on the two components:
\[
\dim(\ls_0)=\dim(\lsk^2)+\dim(\lsk^1)+\dim({\ls^1}_{|Y} \cap {\ls^2}_{|Y}),
\]
where $\lsk^i$ is the kernel of the restriction map $\ls^i \to{\ls^i}_{|Y}$, $i=1,2$.
Upper semi-continuity implies the inequality $\dim(\ls_t)\le\dim(\ls_0)$.

\subsection{Linear systems on $Q\cong\PP^1\times\PP^1$}\label{linear systems on Q}

In order to study the base locus of linear systems on $\PP^3$ through $9$ general points, we want to understand their restrictions to the quadric surface $Q=\ls_{3,2}(1^9)\cong\PP^1\times\PP^1$. The restriction of $\ls=\ls_{3,d}(m_1,\dots,m_9)$ will be the linear series of curves of bidegree $(d,d)$ on $Q$ with $9$ multiple points in general position, that we will denote as $$\ls_{\PP^1\times\PP^1, (d,d)}(m_1,\dots,m_9).$$

Not very much is known about such linear systems: Giuffrida, Maggioni, and Ragusa were
among the first to study linear systems on a quadric surface in \cite{GiMaRa}, see e.g. \cite{GuVT}.
As far as we know the only cases completely classified are those of
double points \cite{VTappendix} and triple points \cite{La}.

The following result allows to transform linear systems of given bidegree 
on the quadric $\PP^1\times\PP^1$ with  multiple base points to  
linear systems on $\PP^2$ with  multiple base points, and vice-versa, 
by means of  \emph{cut-and-sew} of polygons.

The image of $\PP^2$ blown-up at two points via the embedding given by the 
linear system $\ls_{2,d_1+d_2-m}(d_1-m,d_2-m)$ based at two 
torus-invariant points (e.g. two coordinate points) is a toric projective surface 
whose defining polytope is combinatorially equivalent to the 
pentagon obtained by the triangle $(d_1+d_2-m)\Delta$ by cutting two triangles 
$(d_1-m)\Delta$ and $(d_2-m)\Delta$ from two corners,  where
$\Delta$ is the $2$-simplex of $\mathbb{R}^2$, see Figure \ref{polytopes} on the left hand side. 

\begin{figure}[h]
\begin{center}
\includegraphics[width=.9\textwidth]{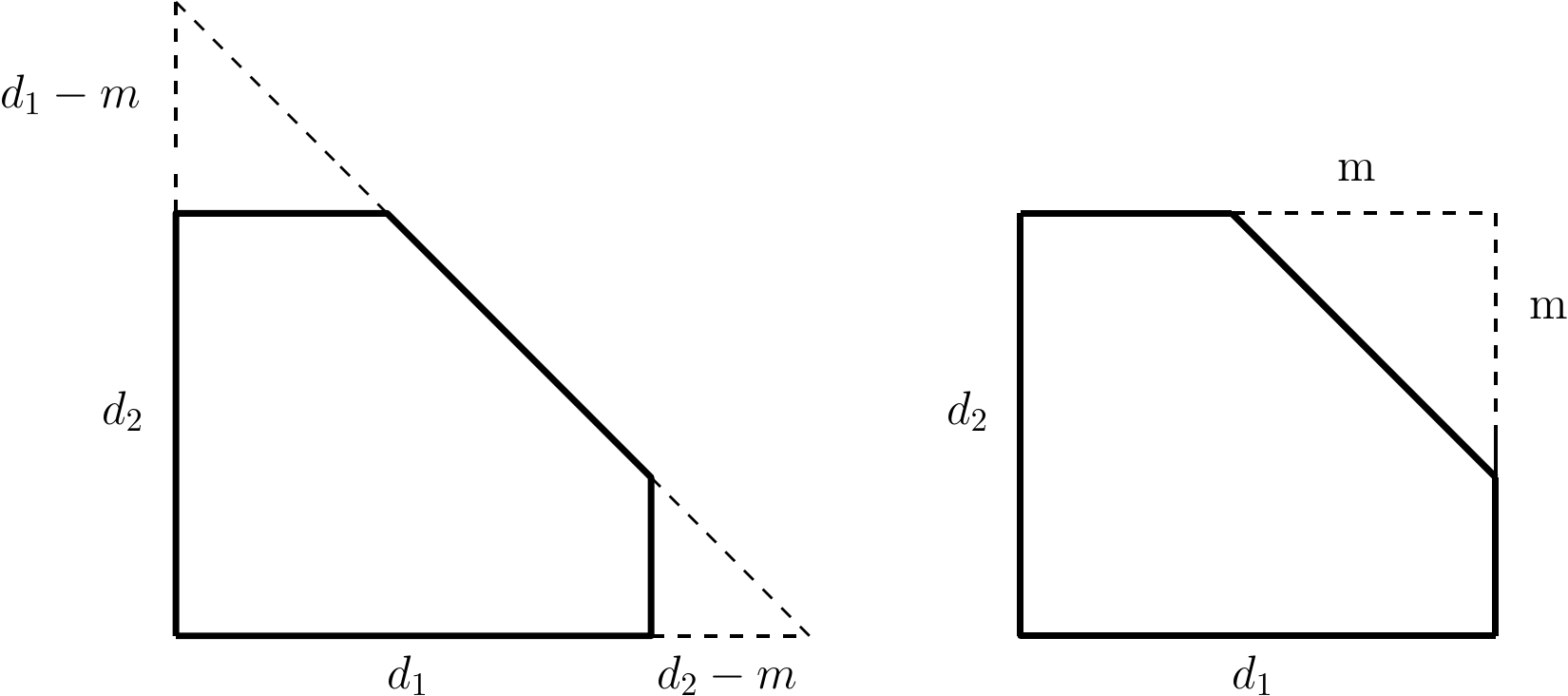}
\end{center}
\caption{Two equivalent polytopes}\label{polytopes}
\end{figure}

Notice that the same polytope can be obtained from the rectangle
 $[0,d_1]\times[0,d_2]\subset\mathbb{R}^2$ by cutting off the triangle $m\Delta$ from a corner.
This interprets the above toric surface  as the embedding of $\PP^1\times\PP^1$ via the linear system  of curves of bidegree $(d_1,d_2)$ with a point of multiplicity $m$, that is
 $\ls_{\PP^1\times\PP^1,(d_1,d_2)}(m),$ see Figure \ref{polytopes} on the right hand side.

In other terms, this is the birational map  that factors in the blow-up 
of $\PP^2$ at two points  and the blow-down of the (-1)-line joining them.

 This proves the following result.

\begin{lemma}\label{equivalence-quadric}
If $m\le d_1,d_2$, then
the following equality holds
$$
\dim(\ls_{\PP^1\times\PP^1,(d_1,d_2)}(m, m_1,\dots,m_s))=
\dim(\ls_{2,d_1+d_2-m}(d_1-m,d_2-m, m_1,\dots,m_s)).$$
\end{lemma}

\begin{remark}
In \cite[Theorem 1.1]{CGG05} the authors show how to convert
 linear systems on products of projective spaces $\PP^{n_i}$ interpolating multiple points
into linear systems in the projective space $\PP^{\sum n_i}$ interpolating 
multiple points and multiple linear subspaces, and back.
We point out that the case $m=0$ in Lemma \ref{equivalence-quadric}  
falls into those equivalences, in the particular case of $\PP^1\times\PP^1$ and $\PP^2$.
\end{remark}

\subsection{Cremona transformations}
\label{cremona}
We recall that the \emph{standard Cremona transformation} of $\PP^n$ is the birational transformation defined by the following rational map:
$$
\textrm{Cr}:(x_0:\dots: x_n) \to (x_0^{-1}:\dots: x_n^{-1}).
$$
This map induces an action on the Picard group of the $n$-dimensional space blown-up at $s$ points.
Let $\ls=\ls_{n,d}(m_1,\dots,m_s)$ be a linear system based on $s$ points in general position; we can assume, without loss of generality, that the first $n+1$ points are the coordinate points.
The Cremona action on $\ls$ is described by the following rule (see for
example \cite{Dolgachev, laface-ugaglia-TAMS}). Set
$$c :=  m_1+\cdots+m_{n+1}-(n-1)d,$$
then 
$$
\textrm{Cr}(\ls)= \ls_{n,d-c}(m_1-c,\dots,m_{n+1}-c, m_{n+2},\dots, m_s)
$$
and $$\dim(\ls)=\dim(\textrm{Cr}(\ls)).$$
We will use this transformations in the cases $n=2,3$ to reduce the computation of the dimension of a linear system $\ls$ to the computation of the dimension of its Cremona transform $\textrm{Cr}(\ls)$ that has lower degree and multiplicities whenever $c>0$.

If $c\le 0$ we will say that the linear system $\ls$ is \emph{Cremona reduced}.

\subsection{Computing with \Mac}
\label{Mac}
In this paper we will need to perform some explicit computation in order to complete our classifications.
In particular the proofs of Proposition \ref{theorem empty 1}, Lemma \ref{ten-points},
Lemma \ref{low-degrees} and Theorem \ref{LUthm} are computer aided.
We perform these computations by means of the computer algebra system \Mac. 
The procedure we use consists essentially in checking that several
square matrices, randomly chosen, have maximal rank.
We work over a field of characteristic
$31991$ and the proofs hold also in characteristic zero.

We use two scripts (one for linear systems in $\PP^2$ and one in $\PP^3$)
available at this url
\url{http://dipmat.univpm.it/~brambilla/NinePointsP3.html},
which allow to compute the dimension and the speciality of a linear system with given degree and multiplicities.

\section{Quasi-homogeneous linear systems $\LL_{3,2m}(m^8,a)$}
\label{section-quasi-homogeneous}

In this section we describe  a class of special linear systems in $\PP^3$ with nine base points for which the quadric surface $Q$ is the only special effect variety. We employ a double degeneration argument,  similar to the one employed in \cite{Po} for linear systems with arbitrary general double points,  that is based on the degeneration of the space described in Section \ref{degeneration technicalities}. The linear system will degenerate into one that has only linear special effect varieties, and that is therefore understood by the results in \cite{BDP,DP}.

Fix non-negative integers  $a,m$. 
Consider the \emph{quasi-homogeneous} linear system in $\PP^3$ 
\begin{equation}\label{2m,m^8,a}
\ls(m,a):=\ls_{3,2m}(m^8,a).
\end{equation} 
The main result of this section is the following. 

\begin{theorem}\label{theorem quasi-homogeneous}
If $1\le a\le m$, the linear system \eqref{2m,m^8,a} satisfies
\begin{align*}
&\dim(\ls(m,a))=m-a+1,\\
&h^1(\ls(m,a))={{a+1}\choose{3}}+{a\choose{2}},
\end{align*}
hence it is special if and only if $2\le a\le m$.
Moreover the only special effect variety for $\ls(m,a)$ is the quadric through nine points which is contained in the base locus with exact multiplicity $a$.
\end{theorem}

It was proved already  in  \cite[Section 5]{laface-ugaglia-BullBelg} that for $a=m$ the linear system \eqref{2m,m^8,a} has one element, that is the $m$-multiple of the quadric through the $9$ points. 
This also implies that if  $a>m$,  the linear system \eqref{2m,m^8,a} is empty. 
The case $a=0$ was proved to be non-special in \cite{volder-laface}.
So the remaining cases to explore are $1\le a\le m-1$; for the sake of completeness
 we include here the
 proof of the case $a=m$ as well.  

Theorem \ref{theorem quasi-homogeneous}, shows that the linear system \eqref{2m,m^8,a} is special with dimension being a linear function of $m$ and $a$.
The only special effect variety is the quadric through the nine points which is contained with multiplicity $a$ in the base locus and moreover, quite surprisingly, 
its contribution to the speciality, namely $h^1(\ls(m,a))$ only depends on the multiplicity of containment of the quadric. 

\begin{remark}\label{quasi-homo-q}
If we define $q(\ls(m,a)):=\chi(\ls(m,a)_{|Q})$ to be the Euler characteristic of the restriction
of $\ls(m,a)$ to the quadric,
(see \eqref{qu} in Section \ref{section Laf-Ug}), then one can easily check the following:
$$
q(\ls(m,a))=1-{{a+1}\choose2}<0 \textrm{ iff }a\ge 2,  \textrm{ and } q(\ls)=0 \textrm{ if }a=1.
$$
\end{remark}

This in particular shows that Theorem \ref{theorem quasi-homogeneous} has the following immediate consequence:

\begin{corollary}
Laface-Ugaglia Conjecture (see Conjecture \ref{LUconj} in Section \ref{section Laf-Ug}) is true for any quasi-homogeneous linear system of the form \eqref{2m,m^8,a}.
\end{corollary}

\subsection{Degeneration of the blown-up $\PP^3$ at $9$ points}
\label{degeneration}

In this section we give a detailed description of the degeneration techniques that we will employ to prove Theorem \ref{theorem quasi-homogeneous}.

\subsubsection{First degeneration}\label{first degeneration}
Consider the trivial family  $\mathcal{V}=\PP^3\times \Delta\to \Delta$ with fibres $V_t\cong\PP^3$, $t\in\Delta$. The blow-up of a point $p_0\in V_0$ produces  a flat morphism  $\mathcal{X}'\to \Delta$ with general fibre $X'_t\cong \PP^3$ and central fibre 
$X'_0=\FF\cup\PP$, where $\FF\cong \textrm{Bl}_{p_0}\PP^3$ is the pull-back of $V_0$ and $\PP\cong\PP^3$ is the exceptional divisor in the total space $\mathcal{X}'$.
The two components $\FF$ and $\PP$ meet transversally along a surface $Y\cong\PP^2$ that, as a divisor,
belongs to the exceptional class of $\FF$ and to the hyperplane class of $\PP$. 
More precisely, if $E_0:=\PP|_\FF$ denotes the exceptional divisor of  $p_0\in V_0$, $H^\FF$ the hyperplane class of $\FF$ and $H^\PP$ that of $\PP$, with   $H^\PP\sim E_0$, we have $\Pic(\F)=\langle H^\F, E_0\rangle$ and
 $\Pic(\PP)=\langle H^\PP\rangle$.

 We choose $7$ general points on $\FF$ and $2$ points on $\PP$ and we consider them as limits of $9$ general points in the general fibre $X'_t$. More precisely, for $t\in\Delta$  let  $\{p_1(t)\dots,p_{9}(t)\}$ be a general collection of points and assume that
 $p_{1}(0),\dots,p_{7}(0)\in\FF$ while $p_{8}(0)$ and $p_{9}(0)\in\PP$.
Consider  $\widetilde{\mathcal{X}}'$  the blow-up of $\mathcal{X}'$ along the 
horizontal curves $\{p_i(t)\}_{t\in\Delta}$, with exceptional divisors $\mathcal{E}_i$,
$i=1,\dots,9$. Denote by $\widetilde{X}'_t$, $t\in\Delta$  the fibres of the new family. Write also  $E_i:=\mathcal{E}_i|_{\widetilde{X}'_t}$, for $t\in\Delta$, $i=1,\dots,9$.
The general fibre is $\widetilde{X}'_t\cong \textrm{Bl}_{p_1,\dots,p_9}(\PP^3)$,
the blow-up of $\PP^3$ at $9$ general points, so that  $\Pic(\widetilde{X}_t)=\langle H,E_1,\dots,\dots, E_9\rangle$. The central fibre is described by
 $\Pic(\F)=\langle H^\F, E_0, E_{1},\dots, E_7\rangle$ and 
 $\Pic(\PP)=\langle H^\PP, E_8,E_9\rangle$, where by abuse of notation $\FF$ and $\PP$ are also the pull-backs in $\widetilde{\mathcal{X}}'$ of the components of $X'_0$.

\subsubsection{Second degeneration}\label{second degeneration}
We  further specialize the points by sending a point from each component 
of $X'_0$ to the intersection. 
More precisely, consider the trivial family $\mathcal{X}'':=X'_0\times \Delta'$
and, on each fibre over $s\in\Delta'$, take a collection of general points 
$\{p_1(s),\dots,p_7(s)\}\subset \FF$ and $\{p_8(s),p_9(s)\}\subset \PP$
such that, on the central fibre, $p_1(0)$ and $p_9(0)\in \FF\cap\PP$. 

Consider  $\widetilde{\mathcal{X}}''$  the blow-up of $\mathcal{X}''$ along the 
horizontal curves $\{p_i(s)\}_{s\in\Delta'}$, with $\mathcal{E}_i$ exceptional divisors, $i=1,\dots,9$. 
The components of the fibres are described by the same Picard groups as the components of $\widetilde{X}'_0$ 
(see Section \ref{first degeneration}).
We use the  symbols $\FF_0$ and $\PP_0$ to denote the pull-backs of the components of the central fibre over $\Delta'$, 
$X''_0$, and the symbol $Y_0$ for their intersection. 
Notice that $Y_0\cong\textrm{Bl}_{p_1,p_9}(\PP^2)$ is a plane blown-up at two points. 

The combination of the two above subsequent degenerations produces a degeneration of 
$\textrm{Bl}_{p_1,\dots,p_9}(\PP^3)$ to the union of blown-up spaces
$\widetilde{X}''_0=\F_0\cup\PP_0$  intersecting along a blown-up plane $Y_0$.

\subsubsection{Intersection table on the central fibre}\label{intersection table}
Notice that, as a divisor on $\FF_0$ (or on $\PP_0$), the surface $Y_0$ 
is represented  by the class $E_0-E_9$ (resp. $H^\PP-E_1$).

One can compute the restrictions of any divisor on $\FF_0$ or on $\PP_0$ to $Y_0$, by means of the following intersection table for the generators of the Picard groups:
\begin{itemize}
\item $H^{\PP}|_{Y_0}=:h$,
\item $E_1|_{Y_0}=:e_1$, 
\item $E_2|_{Y_0}=0$,
\end{itemize} 
and 
\begin{itemize}
\item $H^\F|_{Y_0}=0$, 
\item $E_0|_{Y_0}=-h$, 
\item $E_i|_{Y_0}=0$, $i=3,\dots,8$,
\item $E_9|_{Y_0}=:e_9$.
\end{itemize}
In this notation we have $\Pic(Y_0)=\langle h,e_1,e_9\rangle$.

\subsection{The limit linear system}\label{degeneration of linear system}

Let  $\ls\subset |\mathcal{O}_{\PP^3}(d)|$ be a linear system of degree-$d$ surfaces in $\PP^3$ with nine assigned multiple points in general position. Let $D$ be the corresponding divisor in the blown-up space $\textrm{Bl}_{p_1,\dots,p_9}(\PP^3)$. 

In the notation of Section \ref{first degeneration}, 
consider on $\mathcal{X}'$ the twisted line bundle $\mathcal{O}_{\mathcal{X}'}(d)\otimes \mathcal{O}_{\mathcal{X}'}(-\delta\PP)$.
It restricts to $\mathcal{O}_{\PP^3}(d)$ on $X'_t$ and, for $t=0$, to $\mathcal{O}_\FF(dH_\F-\delta E_0)$ on $\FF$ and to $\mathcal{O}_{\PP}(\delta H^\PP)$ on $\PP$.
By following the first degeneration we can consider the linear system $\ls_t=\ls\subset|\mathcal{O}_{\PP^3}(d)|$ on $X'_t\cong\PP^3$ and its limit $\ls'_0$ on $X'_0$. We denote by $D'_0$ the corresponding divisor class in the blown-up central fibre $\widetilde{X}'_0$.

By following the second degeneration and blowing-up the nine points on each fibre (see Section \ref{second degeneration}), we obtain the  limit divisor $D''_0$ in the blown-up central fibre $\tilde{X}''_0$ that is given by divisors $D^{\F_0}$ and $D^{\PP_0}$ on the two components.
We consider the restriction maps to $Y_0$, $D^{\FF_0}\to R^{\FF_0}:=D^{\FF_0}|_{Y_0}$
and $D^{\PP_0}\to R^{\PP_0}:=D^{\PP_0}|_{Y_0}$ and denote by $\hat{D}^{\F_0}$ and $\hat{D}^{\PP_0}$ the kernels respectively.
Let 
$R_0:=R^{\FF_0}\cap R^{\PP_0}$ denote the intersection of the restricted divisors.

\begin{lemma}\label{formula h0 central divisor}
In the notation of above, for $i\ge0$
we have 
$$
h^i(\widetilde{X}''_0, D''_0)=h^i(\PP_0, \hat{D}^{\PP_0})+h^i(\F_0, \hat{D}^{\F_0})+h^i(Y_0,R_0).
$$
\end{lemma}
\begin{proof}
Notice first of all that the assertion holds if we replace $h^i$ by $\chi$, the Euler characteristic. 
The equality holds for $i=0$ by construction. Indeed, the divisor  $D''_0$ on $\widetilde{X}''_0$, 
or its associated line bundle, is obtained as fibred product of  $D^{\FF_0}$ and $D^{\PP_0}$ over $R_0$, see Section \ref{degeneration technicalities}. 
Finally, since all  cohomology groups with $i\ge2$ vanish, the assertion holds for $i=1$.
\end{proof}

\begin{lemma}\label{semicontinuity}
In the notation of above,  we have
$$h^i(\textrm{Bl}_{p_1,\dots,p_s}(\PP^n),D)\le h^i(\widetilde{X}'_0, D'_0)\le h^i(\widetilde{X}''_0, D''_0), \  i=0,1.$$
\end{lemma}

\begin{proof}
The inequalities hold for $i=0$ by the property of upper semi-continuity of the two degenerations.
As $\chi(\textrm{Bl}_{p_1,\dots,p_9}(\PP^3), D)=\chi(\widetilde{X}'_0, D'_0)=\chi(\widetilde{X}''_0, D''_0)$ 
and all higher cohomology groups vanish, the inequalities hold for $i=1$ as well.
\end{proof}

\begin{remark}
The above construction as well as Lemma \ref{formula h0 central divisor} and Lemma \ref{semicontinuity} is potentially applicable in a more general context for linear systems in any $\PP^n$ and with arbitrary number of points and multiplicities by choosing different specializations and twists, as it was done for instance by the last author in \cite{Po}. Nevertheless, it is not easy to find a good degeneration in general.
\end{remark}

\subsection{Proof of Theorem \ref{theorem quasi-homogeneous}}
In order to prove the theorem, we need the following result. 

\begin{proposition}\label{hatD_F and D_F}
 The following linear systems are non-special with dimension equal to the virtual dimension:
 $\ls_{3,2m}(m+1,m^6,m-1)$, $\ls_{3,2m-1}(m^4,(m-1)^4)$.
\end{proposition}

This can be easily deduced from \cite{volder-laface} where the authors deal with
Cremona reduced linear systems with eight base multiple points.
However, we include the proof here for the sake of completeness.

\begin{proof}
One can easily check that $\vdim(\ls_{3,2m}(m+1,m^6,m-1))=0$.  
The statement is obviously true for $m=1$.
 By performing two subsequent  Cremona transformations of $\PP^3$ (see Section \ref{cremona})
we reduce from $m$ to $m-1$. Hence we conclude 
by induction on $m$.

Similarly, one proves that $\dim(\ls_{3,2m-1}(m^4,(m-1)^4))=
\vdim(\ls_{3,2m-1}(m^4,(m-1)^4))=0$ by induction on $m$.

\end{proof}

\begin{proof}[Proof of Theorem \ref{theorem quasi-homogeneous}]
Let $Q=\ls_{3,2}(1^9)$ be the quadric surface through the nine base points. The obvious inclusion of linear systems $\ls-aQ\subseteq \ls$ implies the inequality $\dim(\ls-aQ)\le \dim(\ls)$. But $\ls-aQ=\ls_{3,2(m-a)}((m-a)^8)$ and, by Proposition \ref{hatD_F and D_F}, 
 $\dim(\ls_{3,2(m-a)}((m-a)^8))=\vdim(\ls_{3,2(m-a)}((m-a)^8))=m-a+1\ge 1$. 
Hence $m-a+1\le \dim(\ls(m,a))$.
 
 We prove the inverse inequality by degeneration. 
Let $D$ denote the
 divisor in $\textrm{Bl}_{p_1,\dots,p_9}(\PP^3)$ corresponding to $\ls(m,a)$:
\begin{align*}
D&=2mH-m\sum_{i=1}^8E_i-aE_9.\\
\intertext{In the notation of Section \ref{degeneration of linear system}, now with $d=2m$, choose $\delta=m$.
In the space $\widetilde{\mathcal{X}}''$ of the second degeneration, we have the 
following divisors on the components of the central fibre $\widetilde{X}''_0$:
}
D^{\F_0}&=2mH^\FF-m E_0-m\sum_{i=1}^{7} E_i,\\
D^{\PP_0}&=mH^\PP-mE_8-aE_9.\\
\intertext{We consider the restriction maps to $Y_0$, 
and we obtain the following kernel divisors
}
\hat{D}^{\F_0}&=2mH^\FF-(m+1) E_0-(m-1)E_1-m\sum_{i=2}^{7} E_i,\\
\hat{D}^{\PP_0}&=(m-1)H^\PP-mE_8-(a-1)E_9.
\end{align*}

Firstly, on the component $\FF_0$ of the central fibre we have the following.
By Proposition \ref{hatD_F and D_F}, we obtain that both $D^{\F_0}$ and 
$\hat{D}^{\F_0}$ are non-special, so that the first cohomology groups vanish; 
moreover, the second is non-effective, namely $h^0=0$. 

Secondly, on the exceptional component of $\widetilde{X}''_0$ we have the following.
The divisor $D^{\PP_0}$ is (only) linearly obstructed and has $h^1(D^{\PP_0})={{a+1}\choose3}$ caused by a line of multiplicity $a$, see \cite{BDP}. 
Moreover the kernel $\hat{D}^{\PP_0}$ is non-effective and has $h^1(\hat{D}^{\PP_0})={{a+1}\choose3}$, see \cite{DP}.

The above implies that both 
$D^{\F_0}$ and $D^{\PP_0}$ cut the complete linear series
 on the intersection $Y_0$ of the components. 
Using Sections \ref{intersection table} and \ref{degeneration of linear system} and the notation there introduced, we have 
$$
R^{\PP_0}=mh-ae_9,\quad
R^{\FF_0}=mh-me_1.
$$
Since $R^{\F_0}$ and $R^{\PP_0}$ meet transversally on $Y_0$, their intersection is given by  
$$R_0=mh-me_1-ae_9.$$
One computes  $h^0(R_0)=m-a+1$ and $h^1(R_0)={{a}\choose2}$, 
the speciality being given 
by a line of multiplicity $a$, see \cite{BDP}.
Finally, by Lemma \ref{formula h0 central divisor} we obtain $h^0(D''_0)=m-a+1$ and
$h^0(D''_0)={{a+1}\choose3}+{{a}\choose2}$.
Now we conclude the proof of the first part of the theorem
by upper semi-continuity, see Lemma \ref{semicontinuity}.

To prove the last sentence of the theorem, we simply notice that the linear system $\ls(m,a)$ splits as follows: $$\ls(m,a)=aQ+\ls_{3,2(m-a)}((m-a)^8).$$
 The second addend in the right hand side is the moving part of $\ls$ and is non-special by Proposition \ref{hatD_F and D_F}. This concludes the proof. 
\end{proof}

\begin{remark}
In the proof of Theorem \ref{theorem quasi-homogeneous}, 
we argued that
 the speciality of $\ls(m,a)$ is given by $aQ$ and equals  the speciality of the limit
 $D''_0$ that, using Lemma \ref{formula h0 central divisor},
is given by a line of multiplicity 
$a$ in the base locus of $\hat{D}^{\PP_0}$ and a line of multiplicity $a$ in the base locus
 of $R_0$.

A geometric interpretation is the following.
 Let us denote by $\ls^{\PP_0,\bf{m}}$ the \emph{matching linear system} 
defined by the matching conditions imposed by $R^{\FF_0}$ to $R^{\PP_0}$, 
so that we have the following exact sequence of sheaves
$$
0\to |\hat{D}^{\PP_0}|\to \ls^{\PP_0,\bf{m}}\to |R_0|\to0.
$$

The emptiness of $|\hat{D}^{\PP_0}|$ implies that the limit linear system 
$|D''_0|$ is the matching linear system; it is of the form $\ls_{3,m}(m,m,a)$ 
where the second and third points, $p_8,p_9$, are general in $\PP_0$, while the first point, $p_1$, is on the intersection and is the one giving the matching.

In particular, if we follow the quadric $Q=\ls(1,1)$ in the degeneration process, by simply setting $m=a=1$ in the above, we see that its limit is given by a matching linear system on $\PP_0$ of the form $\ls_{3,1}(1,1,1)$, based at the points  $p_8,p_9$ and $p_1$ of the central fibre as above.  This linear system has only one element that is the plane spanned by the three points. This plane is
 the special effect variety for the limit of $\ls(m,a)$; it is  contained
 with multiplicity $a$ in the base locus 
 and it contributes by ${{a+1}\choose 3}+{a\choose2}$ to the speciality.
\end{remark}

A \emph{weak base locus lemma} for the quadric $Q$ in the case of $\ls(m,a)$ 
is just an easy 
application of Lemma \ref{equivalence-quadric}. 
  Indeed to prove that $Q$ is contained in the base locus of $\ls(m,a)$ with 
multiplicity at least $a$, it is enough to show that 
for every $m$ and $a$, the restriction $\ls(m,a)|_Q$ is empty and this is equivalent to 
prove that $\ls_{2,3m}(m^9,a)$ is empty, which is a well-known fact.

In the next section we will see that in general to obtain such a result is extremely difficult, mostly
because of the very little knowledge of linear systems on $\PP^1\times\PP^1$.

\section{Homogeneous linear systems  $\LL_{3,2m+1}(m^9)$}
\label{section degree 2m+1}

In this section we consider linear systems with nine points of multiplicity $m$ and degree $2m+1$.
For this class of linear system it is more difficult to understand the relation between the speciality and the 
presence of the quadric as special effect variety. Even to compute
the multiplicity of containment of the quadric in the base locus 
is not an obvious task. The main result of this section is in fact a vanishing result for linear systems in $\PP^2$ which allows to deduce a base locus lemma for the quadric.

Given a linear system $\LL=\LL_{3,d}(m_1,\ldots,m_s)$ with $s\ge9$, let $Q$ be the unique quadric surface through the first nine points. Consider the restriction exact sequence
\begin{equation}\label{restriction-to-quadric}
0\to\LL-Q\to\LL\to\LL_{|Q}\to0.
\end{equation}

The linear system $\LL_{|Q}$ is contained in the linear system 
of the curves of bidegree $(d,d)$ in
 $Q\cong\PP^1\times\PP^1\subset\PP^3$
with nine multiple points and denoted by $\LL_{\PP^1\times\PP^1,(d,d)}(m_1,\ldots,m_9)$ as in Section \ref{linear systems on Q}.

By Lemma \ref{equivalence-quadric} we know that the system 
$\LL_{\PP^1\times\PP^1(d,d)}(m_1,\ldots,m_9)$ has the same dimension 
as the system $\ls_{2,2d-m_1}(d-m_1,d-m_1, m_2,\dots,m_s)$
in $\PP^2$ and in particular the first system is empty if and only if the second one is.

The main part of this chapter is devoted to prove, via degeneration techniques, emptiness results for linear systems 
in $\PP^2$ with ten multiple points.
As a straightforward consequence of Theorem \ref{theorem empty} below, we obtain the following (weak) base locus lemma for the quadric.
Let $\alpha$ be any positive integer.

\begin{theorem}[Quadric base locus lemma]\label{quadric base locus}
Let $\ls=\ls_{3,2m+\alpha}(m^9,m_{10}\ldots, m_s)$ be a non-empty linear system. 
If $m> 9\alpha$, then the quadric $Q$ through the first nine points is contained in the base locus of $\ls$.
\end{theorem}

We remark that
a major difference between the quadric through nine points in $\PP^3$ 
and the linear cycles in $\PP^n$ is in the geometry of their normal bundles. 
For the last ones the normal bundles are toric bundles so we understand
 their cohomology groups \cite{BDP,DP} 
while for the first one the cohomological information is highly non-trivial.

\subsection{Emptiness of linear systems with ten points in $\PP^2$}
\label{section emptiness}

The goal of this section is to find a good bound for $m$ to have emptiness of 
certain linear systems in $\PP^2$. 
More precisely we will prove the following result, which implies Theorem \ref{quadric base locus}.

\begin{theorem}\label{theorem empty}
The linear system $\LL=\LL_{2,3m+2\alpha}((m+\alpha)^2, m^{8})$ is empty for any $m>9\alpha$.
\end{theorem}
{We will prove this result via degeneration techniques similar to the ones introduced in Section \ref{degeneration}. More precisely, we will 
simultaneously degenerate the blown-up projective plane at ten points in general position
and the line bundle $\LL$.}
Even though this technique was applied before in \cite{CDMR, CM3} for homogeneous linear 
systems with ten points, we will present here in detail our approach. 

\subsubsection{The first degeneration}
\label{sec:firstdeg}

{By blowing-up a point in the central fibre of a trivial family of projective planes over a disc, $\Delta$, one obtains a new family, call it $\mathcal{X} \to \Delta.$ The fibre over zero,  $X_0$, decomposes as the union of two surfaces, a projective plane denoted by $\PP$
and the Hirzebruch surface ${\mathbb{F}}_ {1}$, call it $\FF$. In this notation $\PP$ represents the exceptional divisor of the blown-up point, while $\FF$ is the proper transform of the central fibre of the original family. We will denote by $E$ the curve of intersection between $\PP$ and $\FF$.}

{Consider now ten points on the general fibre of the trivial family of planes, such that four of them collide in the zero fibre.
Correspondingly, on the central fibre of $\mathcal{X}$, we place six points on $\FF$ and four points on $\PP$ and we consider them as ten limit points of general points on $X_t$. Blowing-up these ten sections
of $\mathcal{X}$ creates a new family $\mathcal{X}'\to \Delta$. The fibre over zero consists of two surfaces, $\PP_0$ and $\FF_0$, that intersect along a double curve, $E$.
The component $\PP_0$ represents a blown-up plane at four general points, $\FF_0$ represents the blown-up ruled surface ${\mathbb{F}}_ {1}$, at six general points while the double curve $E$ is the negative section on the component ${\mathbb{F}}_ {0}$ and also represents the class of a line on $\PP_0$.
The general fibre $X'_t$ is the blown-up projective plane at ten general points.

\begin{remark}
We point out that colliding four points in the zero fibre works well for the 
analysis of linear systems with ten points. In general, by colliding $s'$ points from a collection of $s$ general points, one produces a degeneration of the blown-up projective plane at $s$ general points, i.e. the general fibre $X'_t$, to the union of two surfaces (for any choice of $s'$ and $s$). The components of the central fibre are: $\PP_0$ that is a blown-up plane at $s'$ general points, and $\FF_0$ that is the blown-up ruled surface ${\mathbb{F}}_ {1}$ at $s-s'$ general points, the two surfaces meeting along a double curve.
\end{remark}

\subsubsection{The second degeneration}
\label{sec:secdeg}
 
This degeneration was first introduced in \cite{CM3}, we provide the construction of the degeneration together with
the limit bundles computation for the sake of completeness. The interested reader should also consult
\cite{CDMR}. We denote by $C$ the unique $(-1)$-curve on $\FF$ passing through six points that 
meets the double curve $E$ in two points $p_1$ and $p_2$, at the form ${\mathcal{L}}_{3}( 2,1^{6})$. We consider the family obtained in Section \ref{sec:firstdeg}, $\mathcal{X}'\to \Delta$, and we blow-up twice the cubic $C$ on $\FF$ and then contract the first exceptional divisor created.
In this way, we will obtain a new family $\mathcal{X}''\to \Delta$ whose general fibre is still a plane blown-up at ten points and whose special fibre over the origin becomes the union of four surfaces. We abuse notations and denote by $\FF$ and $\PP$ the surfaces of the central fibre in the second degeneration and by $\mathbb{S}$ and $\TT$ the exceptional divisors created by the double blow-up of $C.$

{The first blow-up of $C$ in the threefold $\mathcal{X}'$ creates as exceptional divisor a 
Hirzebruch surface ${\mathbb{F}}_1$, that we will denote by $\TT$. 
The rational curve $C$ represents the intersection between $\TT$ and $\FF$. 
In particular, $C$ represents the $(-1)$-curve of $\TT$, while $G_i$ represent the fibre 
class on $\TT$, see Figure 2.}

{The second blow-up of $C$ in $\mathcal{X}$, creates the exceptional divisor $\mathbb{S}$ that is isomorphic to $\PP^1\times\PP^1$
and blows-up the surface $\PP$ twice. Denote by $F_1$ and $F_2$ the exceptional divisors introduced on $\PP$. Notice that the proper transforms of $G_1$ and $G_2$ become
$(-2)$-curves. We abuse the notation and denote by $G_i$ to be these proper transforms with self-intersection $(-2)$.}

{Since the normal bundle of $\mathbb{S}$ has bidegree $(-1, -1)$ one can contract the ruling direction of $\mathbb{S}$. Blowing-down $\mathbb{S}$ will affect the surfaces of the central fibre as follows. On the surface $\FF$ the cubic $C$ will get contracted, $\TT$ will become a projective plane, while on $\PP$ the $(-1)$-curves, $F_1$ and $F_2$, will get identified. In \cite{CDMR, CM3} this operation is 
called a \emph{$2$-throw} of $C$ on $\PP$, see again Figure 2.}

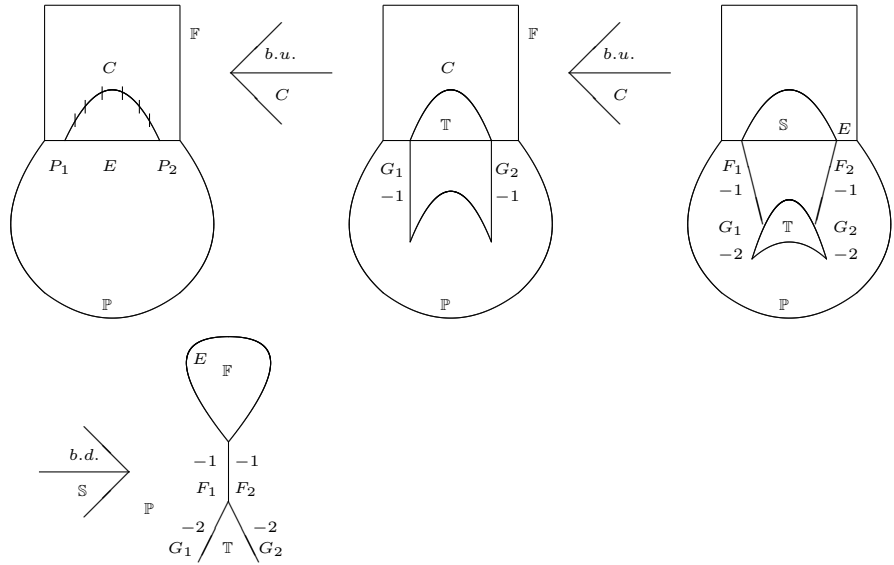
\begin{figure}[h]\label{throw}
\setlength{\unitlength}{0.45mm}
\begin{center}
\begin{picture}(250,100)(0,0)

\put(20,100){\line(1,0){40}}
\put(60,100){\line(0,-1){40}}
\put(60,60){\line(-1,0){40}}
\put(20,60){\line(0,1){40}}

\qbezier(20,60)(0,33)(20,15)
\qbezier(60,15)(80,33)(60,60)
\qbezier(20,15)(40,0)(60,15)

\qbezier(26,60)(40,90)(54,60)

\put(37,76){\line(0,-1){4}}
\put(43,76){\line(0,-1){4}}

\put(32,72){\line(0,-1){4}}
\put(48,72){\line(0,-1){4}}

\put(29,68){\line(0,-1){4}}
\put(51,68){\line(0,-1){4}}

\put(37,80){\mbox{\tiny $C$}}
\put(21,51){\mbox{\tiny $P_1$}}
\put(53,51){\mbox{\tiny $P_2$}}
\put(37,51){\mbox{\tiny $E$}}
\put(37,10){\mbox{\tiny $\PP$}}
\put(63,90){\mbox{\tiny $\FF$}}

\put(85,84){\mbox{\tiny $b.u.$}}
\put(88,72){\mbox{\tiny $C$}}
\put(75,80){\line(1,0){30}}
\put(75,80){\line(1,1){15}}
\put(75,80){\line(1,-1){15}}

\put(120,100){\line(1,0){40}}
\put(160,100){\line(0,-1){40}}
\put(160,60){\line(-1,0){40}}
\put(120,60){\line(0,1){40}}

\qbezier(120,60)(100,33)(120,15)
\qbezier(160,15)(180,33)(160,60)
\qbezier(120,15)(140,0)(160,15)

\qbezier(128,60)(140,90)(152,60)

\put(128,60){\line(0,-1){30}}
\put(152,60){\line(0,-1){30}}

\qbezier(128,30)(140,60)(152,30)

\put(137,80){\mbox{\tiny $C$}}
\put(119,50){\mbox{\tiny $G_1$}}
\put(153,50){\mbox{\tiny $G_2$}}
\put(119,42){\mbox{\tiny $-1$}}
\put(153,42){\mbox{\tiny $-1$}}
\put(137,63){\mbox{\tiny $\TT$}}
\put(137,10){\mbox{\tiny $\PP$}}
\put(163,90){\mbox{\tiny $\FF$}}

\put(220,100){\line(1,0){40}}
\put(260,100){\line(0,-1){40}}
\put(260,60){\line(-1,0){40}}
\put(220,60){\line(0,1){40}}

\qbezier(220,60)(200,33)(220,15)
\qbezier(260,15)(280,33)(260,60)
\qbezier(220,15)(240,0)(260,15)

\qbezier(226,60)(240,90)(254,60)

\put(226,60){\line(1,-4){6.2}}
\put(254,60){\line(-1,-4){6.2}}

\qbezier(229,25)(240,60)(251,25)
\qbezier(229,25)(240,35)(251,25)

\put(220,51){\mbox{\tiny $F_1$}}
\put(253,51){\mbox{\tiny $F_2$}}
\put(219,44){\mbox{\tiny $-1$}}
\put(253,44){\mbox{\tiny $-1$}}
\put(219,33){\mbox{\tiny $G_1$}}
\put(253,33){\mbox{\tiny $G_2$}}
\put(219,25){\mbox{\tiny $-2$}}
\put(253,25){\mbox{\tiny $-2$}}
\put(237,63){\mbox{\tiny $\mathbb{S}$}}
\put(238,33){\mbox{\tiny $\TT$}}
\put(237,10){\mbox{\tiny $\PP$}}
\put(254,62){\mbox{\tiny $E$}}

\put(185,84){\mbox{\tiny $b.u.$}}
\put(188,72){\mbox{\tiny $C$}}
\put(175,80){\line(1,0){30}}
\put(175,80){\line(1,1){15}}
\put(175,80){\line(1,-1){15}}

\end{picture}
\end{center}

\setlength{\unitlength}{0.4mm}
\begin{center}
\begin{picture}(250,70)(0,0)

\put(15,34){\mbox{\tiny $b.d.$}}
\put(18,22){\mbox{\tiny $\mathbb{S}$}}
\put(5,30){\line(1,0){30}}
\put(35,30){\line(-1,1){15}}
\put(35,30){\line(-1,-1){15}}

\qbezier(68,40)(40,75)(68,75)
\qbezier(68,40)(96,75)(68,75)
\put(66,62){\mbox{\tiny $\FF$}}
\put(56,66){\mbox{\tiny $E$}}
\put(68,20){\line(0,1){20}}
\put(68,20){\line(-1,-2){10}}
\put(68,20){\line(1,-2){10}}
\put(40,16){\mbox{\tiny $\PP$}}
\put(57,23){\mbox{\tiny $F_1$}}
\put(70,23){\mbox{\tiny $F_2$}}
\put(48,3){\mbox{\tiny $G_1$}}
\put(78,3){\mbox{\tiny $G_2$}}
\put(66,3){\mbox{\tiny $\TT$}}
\put(56,32){\mbox{\tiny $-1$}}
\put(70,32){\mbox{\tiny $-1$}}
\put(52,10){\mbox{\tiny $-2$}}
\put(76,10){\mbox{\tiny $-2$}}

\end{picture}
\end{center}
\caption{The $2$-throw operation of the double curve $C$.}
\end{figure}

\subsubsection{Degenerating the line bundles.} 

We will now describe limits of line bundles on $\PP^2$ via the double degeneration.
The limit bundles are bundles in the central fibre, $X_{0}''$, of the family $\mathcal{X}''\to \Delta$ that agree on 
the intersection of the double curves.

\begin{remark}\label{infinitely near}
{The double blow-up of $C$ in $\mathcal{X}$ affected the surface $\PP$ by creating two pairs of points that are infinitely near. We will use the notation of \cite{CM3}
$[m_1,m_2]$, to indicate a fat point with multiplicity $m_1$ and an infinitely near fat point with multiplicity $m_2$.
More precisely, we adopt the notation $[m_1,m_2]$ to denote $m_2F_1+m_1(F_1+G_1).$}
\end{remark}

\begin{remark} 
In this section, as in Section \ref{degeneration of linear system}, 
we will describe limits of divisors, and not limits of linear systems.
This difference is emphasized in Remark \ref{infinitely near}. In particular, Proposition \ref{limit} 
should be understood as describing all possible limit divisors $D$ in the linear system 
${\mathcal{L}_{2,d}}((m+\alpha)^2, m^{8})$ on the central fibre $X_0''$. 
However, in order to simplify the language and also to be consistent with notation previously used 
in \cite{CDMR, CM3}, in this section we abuse notations and we use the linear system terminology. 
We must also emphasize now that this degeneration is different than the one we exploited in 
Section \ref{degeneration of linear system}.
More precisely, only the first degeneration of the blown-up projective space $\PP^3$, 
described in Sections \ref{degeneration of linear system} and \ref{sec:firstdeg} coincide. In Section
 \ref{degeneration of linear system} this degeneration was denoted by $\widetilde{\mathcal{X}}'$ while
 in Section \ref{sec:firstdeg} it was denoted by $\mathcal{X}'$. 
However, in  Section 
\ref{degeneration of linear system} the second degeneration was obtained by specializing points on the
 intersection of the two components while in Section \ref{sec:secdeg} the second degeneration is obtained from flopping 
a negative curve. 
In order to highlight this major difference we choose different notations. More precisely even if both 
 represent degenerations of  blown-up projective projective spaces ($\PP^3$ in Section \ref{degeneration of linear system} and $\PP^2$ here)
we will denote them by $\widetilde{\mathcal{X}}''$ and $\mathcal{X}''$ in Sections \ref{degeneration of linear system} and Sections \ref{sec:secdeg} respectively.
\end{remark}

We will now determine all possible limit bundles of the linear system $${\mathcal{L}}_{2,d}( (m+\alpha)^2, m^{8})$$
on the general fibre. 

\begin{itemize}
\item The line bundle on $\PP$ must be of the form ${\mathcal{L}}_{\delta}( m^{4}, [a,b],[a,b])$,
where $\delta, a, b$ represent the twisting parameters.
This line bundle  meets the $4$ times blown-up line $\delta-2a-2b$ times, see Remark 
\ref{geom description} below.
The system on $\FF$ is of the form $\ls_{\FF}=\ls_{2,t}(\delta-2a-2b,y^4,y'^2)$, for some $t,y,y'$. 
The assumption that ${\mathcal{L_{\FF}}}$ doesn't meet the cubic $C$ implies that the degree of ${\mathcal{L_{\FF}}}$ has to be even, write $t=2e$. 
Indeed $0=\ls_{\FF} \cdot C= 3t-2\delta+4a+4b-4y-2y'$.

Moreover, because $y=m-a-b$ and  $y'=m-a-b+\alpha$, then one can check that
$$\delta=3e-3m+5a+5b-\alpha.$$
\item Consider the intersection of $\mathbb{S}$ and $\FF$ that is a fibre on $\mathbb{S}$ and the cubic on $\FF.$
Note that $\mathcal{L_{\mathbb{S}}}$ is a horizontal bundle so it must have bidegree $(b,0)$. Moreover,
${\mathcal{L_{\TT}}}$  meets a fibre $a-b$ times and does not meet the negative section $B$.
Hence $${\mathcal{L_{\TT}}}={\mathcal{L}}_{2,a-b}.$$
\item The last parameter to be determined is $e$. We compute it by observing that the limit bundle should have degree $d$, that is the degree on the bundle on the general fibre. By pulling-back a line in the plane we get a line on $\PP$, a fibre on $\FF$, a fibre on $\TT$ and a fibre on $\mathbb{S}$. Therefore, the intersection number with all the bundles from above will have to add up to $d$. We obtain
$$e=\frac{d-3(a+b)}{2}.$$ 
\end{itemize}
Solving this system of linear equations we obtain

$$
{\mathcal{L_{\FF}}}={\mathcal{L}}_{2,d-3a-3b}\left(\frac{3d}{2}-3m-\frac{3(a+b)}{2}-\alpha, (m-a-b)^{4}, (m-a-b+\alpha)^{2}\right).
$$

The surface $\mathbb{S}$ will be contracted in the ruling direction. This last blow-down will affect the surface $\FF$ by contracting the cubic $C$ to a point by performing a series of Cremona transformations to ${\mathcal{L_{\FF}}}$, see Section \ref{cremona}. 
After contracting the surface $\mathbb{S}$ the bundle on $\FF$ becomes
$${\mathcal{L_{\FF}}}={\mathcal{L}}_{2,3m-\frac{d}{2}-\frac{3(a+b)}{2}+\alpha}\left(0, \left(2m-\frac{d}{2}-\frac{a+b}{2}+\alpha\right)^4, \left(2m-\frac{d}{2}-\frac{a+b}{2}\right)^2\right).$$
The zero multiplicity of ${\mathcal{L_{\FF}}}$ represents the image of the cubic after the Cremona transformations. Since we are contracting the surface $\mathbb{S}$ we will simply ignore this multiplicity.
We recall that contracting the cubic on $\FF$ will also affect the surface on $\PP$ by identifying the two last  $(-1)$-curves created on $\PP$, namely $F_{1}$ and $F_{2}.$

For the future analysis we will work with the normalization of $\PP,$ so we will consider $F_{1}$ and $F_{2}$ disjoint as before.
We obtain the following result.

\begin{proposition}\label{limit}
All limits of the bundle ${\mathcal{L}_{d}}((m+\alpha)^2, m^{8})$ are of the following form, for some choice of the parameters $a$ and $b$:
\begin{itemize}
\item ${\mathcal{L_{\PP}}}={\mathcal{L}}_{2,\frac{3d}{2}-3m+\frac{a+b}{2}-\alpha}\left(m^{4}, [a,b],[a,b]\right)$,
\item ${\mathcal{L_{\FF}}}={\mathcal{L}}_{2,3m-\frac{d}{2}-\frac{3(a+b)}{2}+\alpha}\left((2m-\frac{d}{2}-\frac{a+b}{2}+\alpha)^4, (2m-\frac{d}{2}-\frac{a+b}{2})^2\right)$,
\item ${\mathcal{L_{\TT}}}={\mathcal{L}}_{2,a-b}$.
\end{itemize}
\end{proposition}

\begin{remark}\label{geom description}
We describe here how our degeneration method works.
The choice of the number of points that one collides in the central fibre in the first degeneration (in this case is four, see Section \ref{sec:firstdeg}) determines at each step of the degeneration the 
$(-1)$-curves that one must flop in order to obtain better ratios for proving emptiness or non-speciality results. In the analysis of linear systems with ten points both approaches of proving emptiness or non-speciality lead to the same degeneration, see \cite{CDMR} and \cite{CM3}. 

Namely, in any degeneration one computes a fixed limit ratio $\frac{d}{m}$ such that 
for any choice of the twisting parameters, all line bundles on the central fibre are 
non-empty. But this ratio satisfies the Nagata bound, so one of the linear systems 
is special. According to Segre-Harbourne-Gimigliano-Hirschowitz Conjecture 
this effective linear system is special because of the existence of a negative curve. 
In the first degeneration this curve is precisely the special cubic passing through 
seven points, denoted by $C$, that distinguishes by splitting off the line bundle on 
$\PP$. To improve the limit ratio, $\frac{d}{m}$, one needs to flop the curve creating 
speciality.

In general, if the curve creating speciality of one of the linear systems of the central fibre intersects the double curve $E$ once, then one performs a one-throw as explained in \cite{CM3}. In other words, by blowing-up the curve creating speciality, the exceptional divisor introduced is a ruled surface isomorphic to $\PP^1\times\PP^1$, so one can contract the other ruling. 

However in our case, the special curve is the cubic $C$ intersecting the double curve $E$ twice. It follows by the general intersection theory that the exceptional divisor created after the first blow-up, denoted by $\TT$, is a Hirzebruch surface $\FF_1$ that can not be contracted. So a second blow-up is necessary creating two infinitely near points; 
the new exceptional divisor, denoted by $\mathbb{S}$, is the ruled surface $\PP^1\times\PP^1$. This can be seen by intersection theory of surfaces in $\PP^3$. This affects the double curve of intersection, transforming $E$ to the strict transform of a line blown-up four times.
Finally, the fibre direction of $\mathbb{S}$ can be blown-down. This blow-down simplifies the geometry of $\FF$ but it increases the difficulty of the study of the linear system on $\PP$.

We would like to point our that each degeneration is uniquely determined by the number of points we decide to collide in the first step, the degree and multiplicities of the linear system. In this case the same degeneration as in \cite{CDMR} and \cite{CM3} can be applied, but for the sake of simplicity the computations that lead to this degeneration were omitted.
\end{remark}

We study now the effectivity of ${\mathcal{L_{\PP}}}$ and ${\mathcal{L_{\FF}}}$ for $d=3m+2\alpha$ and $m\ge 8\alpha$. 
Notice that by substituting $d=3m+2\alpha$ we obtain the following bundles on $\PP$ and $\F$:
\begin{align*}
{\mathcal{L_{\PP}}}&={\mathcal{L}}_{2,\frac{3m}{2}+\frac{a+b}{2}+2\alpha}\left( m^{4}, [a,b],[a,b]\right),\\
{\mathcal{L_{\FF}}}&={\mathcal{L}}_{2,\frac{3m}{2}-\frac{3(a+b)}{2}}\left( \left(\frac{m}{2}-\frac{a+b}{2}\right)^4, \left(\frac{m}{2}-\frac{a+b}{2}-\alpha\right)^2\right).
\end{align*}

\begin{remark}\label{F and T} The following two statements are obvious.
\begin{itemize}
\item
The linear system on $\TT,$ ${\mathcal{L_{\TT}}}={\mathcal{L}}_{2,a-b},$ is nonempty if and only if $a\geq b.$
\item 
The linear system on $\FF,$ ${\mathcal{L_{\FF}}}$ is non-empty if any only if $a+b\leq m.$
\end{itemize}
\end{remark}

We will now analyse the linear system on $\PP$.
We denote by $Q_{i}$ the four quartics ${\mathcal{L}}_{4}(2^3,1,[1,1]^{2})$ on $\PP$ and we see that these $(-1)$-curves split off the system if $m\geq 8\alpha$.
Indeed,
$${\mathcal{L}_{\PP}} Q_{i}={\mathcal{L}}_{2,\frac{3m}{2}+\frac{a+b}{2}+2\alpha}( m^{4}, [a,b],[a,b])\ {\mathcal{L}}_{2,4}(2^3,1,[1,1]^{2})=$$
$$4\left(\frac{3m}{2}+\frac{a+b}{2}+2\alpha\right)-3\cdot 2\cdot m-m-2a-2b=
8\alpha-m.$$ 
We further apply a series of four Cremona transformations to the linear system $\LL_{\PP}$ (which contains $8$ base points $p_1,\ldots,p_8$)
based respectively at the points 
$\{p_1,p_2,p_3\}$, $\{p_4,p_5,p_8\}$, $\{p_4,p_6,p_7\}$ and $\{p_1,p_2,p_3\}$.
Note that this series of Cremona 
transformations contracts the four quartics to a point at the same time. 
$$\textrm{Cr}({\mathcal{L}}_{\PP})=\mathcal{L}_{2,\frac{a+b}{2}-\frac{5m}{2}+18\alpha}((8\alpha-m)^4,[a-m+4\alpha,b-m+4\alpha], [a-m+4\alpha, b-m+4\alpha]).$$

For $m\geq 8\alpha$ 
the exceptional divisors corresponding to the first four points are $(-1)$-curves that split off the system. 
These exceptional divisors represent the four quadrics; we will remove them and forget the zero multiplicities created. The
residual system is
$${\mathcal{L'}}_{\PP}=\mathcal{L}_{2,\frac{a+b}{2}-\frac{5m}{2}+18\alpha}([a-m+4\alpha,b-m+4\alpha], [a-m+4\alpha, b-m+4\alpha])$$
It is obvious that ${\mathcal{L}}_{\PP}$ is empty if and only if ${\mathcal{L'}}_{\PP}$ is empty.

We are now ready to prove the main result of this section.

\begin{proof}[Proof of Theorem \ref{theorem empty}]
We want to prove that $\LL=\LL_{2,3m+2\alpha}((m+\alpha)^2, m^{8})$ is empty for $m>9\alpha$.

We assume by contradiction that there are some values of the 
parameters $a$ and $b$ for which both linear systems ${\mathcal{L}}_{\PP}$ and ${\mathcal{L}}_{\FF}$ 
are non-empty in the central fibre of the degeneration. 
If ${\mathcal{L}}_{\PP}$ is non-empty then the degree of ${\mathcal{L'}}_{\PP}$ is positive. 
In particular
$$a+b\geq 5m-36\alpha.$$
On the other hand, since ${\mathcal{L}}_{\FF}$ is non-empty, 
by Remark \ref{F and T}, we must have
$$a+b\leq m.$$

These two inequalities lead to a contradiction, hence 
the linear system $\LL_{2,3m+2\alpha}((m+\alpha)^2, m^{8})$ is empty. 
\end{proof}

In particular, Theorem \ref{theorem empty} gives the following consequence.

\begin{proposition}\label{theorem empty 1}
If $m\geq 8$, then the linear system $\LL_{2,3m+2}((m+1)^2, m^{8})$ is empty.
\end{proposition}

\begin{proof}
First, we check cases $m=8,9$ by computer. 
For $m\geq 10$ we apply Theorem \ref{theorem empty} with $\alpha=1$.
\end{proof}

\subsection{Classification of homogeneous linear systems  $\LL_{3,2m+1}(m^9)$}
We are now in position to prove the complete classification of homogeneous 
linear systems in $\PP^3$ of degree $2m+1$ and with nine points.

\begin{theorem}
\label{degree 2m+1}
A linear system $\LL=\LL_{3,2m+1}(m^9)$ is special if and only if $m\ge9$. In particular we have:
\begin{itemize}
\item $\dim(\ls_{3,2m+1}(m^9))=\vdim(\ls_{3,2m+1}(m^9))$ for $m\le 8$;
\item $\dim(\ls_{3,2m+1}(m^9))=60$ for $m\ge 7$;
\item the quadric $Q$ through the nine base points is in the base locus of $\LL$ with multiplicity $m-7$, for any $m\ge 8$.
\end{itemize}
\end{theorem}

Before proceeding with the proof of this classification result, we state the following lemma, 
that is an easy consequence of Proposition \ref{theorem empty 1}.
\begin{lemma}\label{ten-points}
The linear system $\LL_{2,3m+2}((m+1)^2,m^8)$ satisfies:
$$\dim(\LL_{2,3m+2}((m+1)^2,m^8))=\chi(\LL_{\PP^1\times\PP^1,(2m+1,2m+1)}(m^9))$$
for $m\le 8$ and it is empty for $m\ge8$.
\end{lemma}

\begin{proof}
We check by computer the statement for $m\le 7$. 
For $m\ge 8$ we use Proposition \ref{theorem empty 1}.
\end{proof}

By Lemma \ref{equivalence-quadric}, the previous lemma has the following straightforward consequence.

\begin{corollary}\label{syst-quadric}
The linear system $\LL_{\PP^1\times\PP^1,(2m+1,2m+1)}(m^9)$ is non-special for every $m\ge1$ and it is empty for $m\ge8$.
\end{corollary}

\begin{proof}[Proof of Theorem \ref{degree 2m+1}]
The restriction exact sequence
\eqref{restriction-to-quadric}
gives in this case:
$$
0\to\LL_{3,2(m-1)+1}((m-1)^9)\to\LL_{3,2m+1}(m^9)\to\LL_{\PP^1\times\PP^1,(2m+1,2m+1)}(m^9)\to0.
$$
By using induction on $m\ge1$ and Corollary \ref{syst-quadric}, we deduce that 
the linear system is non-special if and only if $m\le 8$ 
(notice that if $m=8$ we have $\chi(\LL_{\PP^1\times\PP^1,(17,17)}(8^9))=0$). 
In order to prove that the quadric is contained in the base locus of $\LL$ with multiplicity $m-7$, 
it is enough to use Corollary \ref{syst-quadric} and to notice that 
$\dim(\LL_{3,15}(7^9))\neq\dim(\LL_{3,13}(6^9))$.
\end{proof}

A straightforward consequence of Theorem \ref{degree 2m+1} is the following:

\begin{corollary}
Conjecture \ref{LUconj} holds  for any 
homogeneous linear system with nine points of multiplicity $m$ and degree $d\le 2m+1$.
\end{corollary}

\section{Proof of Laface-Ugaglia Conjecture for linear systems with 9 points and multiplicities bounded by $8$}\label{proof LU}
\label{section Laf-Ug}

Let $\LL = \LL_{3,d} (m_1 ,\ldots , m_9)$ be the linear system of degree
 $d$ hypersurfaces
of $\PP^3$ with $9$ general multiple points of multiplicities $m_1,\dots,m_9$.
In this section we will assume that  $d\ge m_1\ge m_2\ge\ldots\ge m_9$.
Let $Q=\LL_{3,2}(1^9)$ be the 
unique quadric surface through the nine base points.  
We adopt the following notation
\begin{equation}
\label{qu}
q(\LL)=\chi(\LL_{|Q})=(d+1)^2-\sum_{i=1}^9\binom{m_i+1}{2}.
\end{equation}

Laface and Ugaglia formulated their conjecture in \cite[Conjecture 4.1]{laface-ugaglia-TAMS} and \cite[Conjecture 6.3]{laface-ugaglia-standard}.
Following the definition of linear speciality introduced in \cite{BDP}, we can reformulate 
this conjecture in the following way.

\begin{conjecture}[Laface-Ugaglia Conjecture]\label{LUconj}
Given a Cremona reduced linear system $\LL$ in $\PP^3$, we have 
\begin{enumerate}
\item
if $q(\LL)\le 0$, then $\dim(\LL)=\dim(\LL-Q)$;
\item
if $q(\LL)> 0$, then $\LL$ is linearly non-special.
\end{enumerate}
\end{conjecture}

\begin{remark}\label{easy}
Since $\LL$ is Cremona reduced, i.e. $m_1+m_2+m_3+m_4\le 2d$, it does not contain
any plane in the base locus. Hence Conjecture \ref{LUconj} says that if $q(\LL)>0$, 
then $\LL$ is special if and only if $m_1+m_2-d\ge 2$ and in this case:
$$\dim(\LL)=\ldim(\LL)=\chi(\LL)+\sum_{i,j}\binom{m_i+m_j-d+1}3,$$
where $\ldim$ denotes the affine linear dimension, see \cite[Definition 1.2]{DP}.
\end{remark}

\begin{remark}\label{q}
If $q(\LL)\le 0$ and $\dim(\LL)=\dim(\LL-Q)$, 
from the exact sequence \eqref{restriction-to-quadric}
we obtain that
$$
\hh^1(\LL)=\hh^1(\LL-Q)-q(\LL).
$$
This means that the quadric $Q$ is a special effect surface for the linear 
system $\ls$.
\end{remark}

\begin{remark}
We point out that a quadric surface in the base locus can give speciality even if it is 
contained with multiplicity one.
Consider for instance the linear system $\LL=\LL_{3,8}(4^7,3^2)$ for
which $\dim(\LL)=6$, $\hh^1=-q(\LL)=1$. 
This system contains in its base locus the quadric $Q$ through the nine points, but does not contain $2Q$.

This behaviour is different from the case of linear special effect varieties, for which any linear cycle of dimension $l$ contributes to the speciality only if its multiplicity in the base locus is at least $l+1$.
\end{remark}

\begin{remark}\label{q-in-general}
Notice that when a linear system $\LL$ has a quadric 
surface as special effect variety, computing $\hh^1(\LL)$ is quite difficult in general.
In fact the quasi-homogeneous systems classified in Section \ref{section-quasi-homogeneous}
form a very special family for which we understand completely the situation, 
but this is not the case in general. 

Let $\LL$ be a linear system with $q(\LL)\le0$.
Assume that
\begin{itemize}
\item
 $q(\LL-k Q)\le0$ for any $0\le k\le \overline{k}$
and $q(\LL-(\overline{k}+1)Q)>0$,
\item $\LL-k Q$ restricts to
non-special linear systems on the quadric $Q$, for any $0\le k\le \overline{k}$.
\end{itemize}

Then, by using Remark \ref{q}, we get the following formula:
\begin{equation}\label{formula h1}
\hh^1(\LL)=-\sum_{k=0}^{\overline{k}}q(\LL-k Q).
\end{equation}

By using \eqref{formula h1} in  the case of quasi-homogeneous systems $\LL_{3,2m}(m^8,a)$ 
and using Remark \ref{quasi-homo-q}, we recover exactly the formula
$\hh^1(\LL) = \binom{a+1}{3}+\binom{a}2$ 
of Theorem \ref{theorem quasi-homogeneous}. In this case $\overline{k}= a$.

The problem in general is to determine the value of $\overline{k}$.
Let us see an example:
if $\LL=\LL_{3,13}(8,6^8)$, then $q(\LL)=-8$, $q(\LL-q)=-4$,  $q(\LL-2Q)=-1$, while $q(\LL-3Q)=1>0$, 
hence we have $\hh^1(\LL)=8+4+1=13$ and in this case $\overline{k}=2$.
\end{remark}

\begin{remark}\label{subsystems}
Given two vectors $v=(m_1,\ldots,m_s)$ and $v'=(m'_1,\ldots,m'_s)$ in $\NN^s$, we write $v'\le v$ if
and only if $m'_i\le m_i$ for any $1\le i\le s$.

It is easy to see  that if a linear system $\LL_{n,d}(v)$ is non-special and non-empty,
then also any linear system $\LL_{n,d}(v')$ is non-special and non-empty for any vector $v'\le v$.
\end{remark}

Now we establish Laface-Ugaglia Conjecture for any linear system with nine points of multiplicities bounded by $8$.
We start with a lemma whose proof is essentially computational.

\begin{lemma}\label{low-degrees}
If a linear system $\LL=\LL_{3,d}(m_1,\ldots,m_9)$ is such that
$m=\max(m_i)\le 8$ and $d< 2m$, then it satisfies Conjecture \ref{LUconj}.
\end{lemma}
\begin{proof}
First of all it is clear that if $d<m$ the system is empty, so we assume $d\ge m$.
Assume that $\LL$ is Cremona reduced, that is 
\begin{equation}\label{crem}
m_1+m_2+m_3+m_4\le 2d.
\end{equation}

Now if $d=m$, then by \eqref{crem} we have that $m_1=m$ and $m_2<m$. 
Therefore by applying \cite[Theorem 5.3]{BDP}, we have that if $\sum_{i=1}^9 m_i\le 3d +2$ then $\LL$ is linearly non-special.
Hence we can also assume
\begin{equation}\label{allmult-degreem}
\sum_{i=1}^9 m_i> 3d +2.
\end{equation}

For any $m\le8$, only the following systems satisfy
conditions \eqref{crem} and \eqref{allmult-degreem}: $\LL_{3,6}(6,2^8)$ and $\LL_{3,7}(7,3,2^7)$.
It is easy to check that these two systems are linearly non-special.

Assume now that $d\ge m+1$.
We know, by \cite{ballico-brambilla-caruso-sala}, that Laface-Ugaglia Conjecture 
is true for any linear system with multiplicities 
bounded by $5$. So we can assume $6\le m\le 8$.

Moreover, by applying again \cite[Theorem 5.3]{BDP}, 
we have that if $\sum_{i=1}^9 m_i\le 3d +3$ then $\LL$ is linearly non-special.
Hence we can also assume
\begin{equation}\label{allmult}
\sum_{i=1}^9 m_i> 3d +3.
\end{equation}

Now we list all the possible linear systems which satisfy conditions \eqref{crem} and \eqref{allmult}, for any $5\le m\le 8$ and any $m+1\le d\le 2m-1$.
Then we prove that all the cases in the list satisfy the conjecture using the following procedure. For any degree we start to check the
 cases $\LL=\LL_{3,d}(m_1,\ldots,m_8)=\LL_{3,d}(v)$ for the largest vectors $v$.
We compute $\dim(\LL)$ by means of the computer 
system \Mac\ as explained in Section \ref{Mac}.

If $\LL$ is non-special and non-empty, then by Remark \ref{subsystems}, 
also the linear system $\LL_{3,d}(v')$ is non-special and non-empty,
 for any vector $v'\le v\in\ZZ^9$, hence we greatly reduce the number of cases to be checked.

If $\LL$ is linearly non-special, then we apply \cite[Lemma 5.5 and Remark 5.6]{BDP} and \cite[Theorem 1.2]{Chandler}
and we obtain again that any system
$\LL_{3,d}(v')$, for $v'\le v\in\ZZ^9$, is linearly non-special.
Hence we further reduce the number of cases to be checked and 
we obtain at the end the lists contained in Tables 1, 2, 3. 
Notice that in the tables the special and linearly non-special systems are marked with  $*$. 
By applying this procedure we complete the proof of the lemma.
\end{proof}

We give now the main result of this section: 
\begin{theorem}\label{LUthm}
Conjecture \ref{LUconj} is true for any linear system $\LL_{3,d}(m_1,\ldots,m_9)$ such that
$m=\max(m_i)\le 8$.
\end{theorem}

\begin{proof}
If the degree $d\le 2m-1$ the result follows from Lemma \ref{low-degrees}.

If $d=2m$, by Theorem \ref{theorem quasi-homogeneous}, 
we know that the quasi homogeneous linear systems $\LL_{3,2m}(m^8,a)$ are special 
if and only if $2\le a\le m$ and they satisfy Conjecture \ref{LUconj}.
Arguing as in Lemma \ref{low-degrees}, in order to complete the proof we need to check 
all linear systems satisfying \eqref{crem} and \eqref{allmult} for any $6\le m\le 8$.

The list of these cases (reduced by Remark \ref{subsystems}) is contained in Table 4 
and we checked all of them by computer.

Now if $d\ge 2m+1$, by Theorem \ref{degree 2m+1}  the linear system $\LL_{3,2m+1}(m^9)$ is non-special
and non-empty. Hence any homogeneous linear system $\LL_{3,d}(m^9)$ 
for $d\ge 2m+1$ is also non-special and non-empty.

Finally we deduce that any (non-homogeneous) linear system $\LL_{3,d}(m_1,\ldots,m_9)$ with $m_i\le m$ 
is non-special and non-empty, by Remark \ref{subsystems}.
This completes the proof.
\end{proof}

\subsection{Future directions} 
We conclude this paper by pointing out  possible future directions (both theoretical and computational)
 in  establishing Laface-Ugaglia Conjecture for nine points.
On the one hand, one can introduce further degenerations of $\PP^2$ in order to obtain a better bound in the base locus lemma,  
Theorem \ref{quadric base locus}.
On the other hand, the combination of the results of Section \ref{section degree 2m+1} and 
of similar computer-based computations as the one performed in this section could
 improve the bound on the multiplicities of Theorem \ref{LUthm}. 

\newpage
\subsection{Tables}
The linear systems marked with $*$ in Table \ref{table-m=6} , Table \ref{table-m=7} and Table \ref{table-m=8}
are linearly non-special, namely their dimension equals the linear expected dimension. 
All other linear systems have  the quadric surface through nine points as
special effect component, namely it splits off the system and gives speciality.
\ 
\begin{table}[!ht]
\centering

\caption{The case $m=6$}\label{table-m=6}
\begin{tabular}{|c|c|c|c|c|c|}
\hline
degree&$(m_1,m_2,m_3,m_4,m_5,m_6,m_7,m_8,m_9)$& $q$ & $\hh^0$ & $\hh^1$ \\
\hline
11 &(6, 5, 5, 5, 5, 5, 5, 5, 5)& 3& 28&0\\
11&(6, 6, 5, 5, 5, 5, 5, 5, 4)& 2& 22&0\\
11&(6, 6, 5, 5, 5, 5, 5, 5, 5)& -3& 10&3\\
11&(6, 6, 6, 4, 4, 4, 4, 4, 4)& 21& 76&0\\
\hline
10&(6, 5, 5, 4, 4, 4, 4, 4, 4)& 10& 40&0\\
10&*(6, 6, 4, 4, 4, 4, 4, 4, 4)& 9&35 &1\\
10&*(6, 6, 5, 3, 3, 3, 3, 3, 3)&28 &80 &1\\
\hline
9&(6, 4, 4, 4, 4, 4, 4, 4, 3)& 3& 14&0\\
9&(6, 4, 4, 4, 4, 4, 4, 4, 4)& -1&5&1\\
9&*(6, 5, 4, 3, 3, 3, 3, 3, 3)& 18&50&1\\
9&*(6, 6, 3, 3, 3, 3, 3, 3, 3)& 16&42&4\\
\hline

8&*(6, 4, 3, 3, 3, 3, 3, 3, 3)& 8& 20&1\\
\hline
\end{tabular}
\end{table}

\begin{table}[!ht]
\centering
\caption{The case $m=7$}\label{table-m=7}
\begin{tabular}{|c|c|c|c|c|c|}
\hline
degree&$(m_1,m_2,m_3,m_4,m_5,m_6,m_7,m_8,m_9)$& $q$ & $\hh^0$ & $\hh^1$ \\
\hline
13&(7, 6, 6, 6, 6, 6, 6, 6, 6)& 0& 28&0\\
13&(7, 7, 6, 6, 6, 6, 6, 5, 5)& 5& 42&0\\
13&(7, 7, 6, 6, 6, 6, 6, 6, 4)& 4& 36&0\\
13&(7, 7, 6, 6, 6, 6, 6, 6, 5)& -1&22&1\\
13&(7, 7, 6, 6, 6, 6, 6, 6, 6)& -7& 10&10\\
13&(7, 7, 7, 5, 5, 5, 5, 5, 5)& 22& 98&0\\
\hline
12&(7, 6, 6, 5, 5, 5, 5, 5, 5)& 9& 49&0\\
12&*(7, 7, 5, 5, 5, 5, 5, 5, 5)& 8& 43&1\\
12&*(7, 7, 6, 4, 4, 4, 4, 4, 4)& 32& 112&1\\
\hline
11&(7, 5, 5, 5, 5, 5, 5, 5, 4)& 1& 15&0\\
11&(7, 5, 5, 5, 5, 5, 5, 5, 5)& -4&5&5\\
11&*(7, 6, 5, 4, 4, 4, 4, 4, 4)& 20&70&1\\
11&*(7, 6, 6, 3, 3, 3, 3, 3, 3)& 38&110&2\\
11&*(7, 7, 4, 4, 4, 4, 4, 4, 4)& 18&60&4\\
11&*(7, 7, 5, 3, 3, 3, 3, 3, 3)& 37&105&4\\
\hline

10&*(7, 5, 4, 4, 4, 4, 4, 4, 4)& 8& 28& 1\\
10&*(7, 5, 5, 3, 3, 3, 3, 3, 3)& 27& 74& 2\\
10&* (7, 6, 4, 3, 3, 3, 3, 3, 3)& 26& 70& 4\\
10&*(7, 7, 3, 3, 3, 3, 3, 3, 3)& 23& 58& 10\\
\hline

9&*(7, 4, 4, 3, 3, 3, 3, 3, 3)& 16& 38& 2\\ 
9&*(7, 5, 3, 3, 3, 3, 3, 3, 3)& 15& 35& 4\\ 
\hline
8&*(7, 3, 3, 3, 3, 3, 3, 3, 3)& 5& 9& 8\\
\hline
\end{tabular}
\end{table}

\clearpage

\begin{table}[!ht]
\centering
\caption{The case $m=8$}\label{table-m=8}
\begin{tabular}{|c|c|c|c|c|c|}
\hline
degree&$(m_1,m_2,m_3,m_4,m_5,m_6,m_7,m_8,m_9)$& $q$ & $\hh^0$ & $\hh^1$ \\
\hline
15&(8, 7, 7, 7, 7, 7, 7, 7, 6)& 3& 52&0\\
15&(8, 7, 7, 7, 7, 7, 7, 7, 7)& -4& 28&4\\
15&(8, 8, 7, 7, 7, 7, 7, 6, 6)& 2& 44&0\\
15&(8, 8, 7, 7, 7, 7, 7, 7, 5)& 1& 37&0\\
15&(8, 8, 7, 7, 7, 7, 7, 7, 6)& -5& 22&6\\
15&(8, 8, 7, 7, 7, 7, 7, 7, 7)& -12&10&22\\
15&(8, 8, 8, 6, 6, 6, 6, 6, 6)& 22& 120&0\\
\hline
14&(8, 7, 7, 6, 6, 6, 6, 6, 6)& 7& 56&0\\
14&*(8, 8, 6, 6, 6, 6, 6, 6, 6)& 6& 49& 1\\
14&*(8, 8, 7, 5, 5, 5, 5, 5, 5)& 35& 147& 1\\
14&*(8, 8, 8, 4, 4, 4, 4, 4, 4)& 57& 203& 3\\
\hline
13&(8, 6, 6, 6, 6, 6, 6, 5, 5)& 4& 34&0\\
13&(8, 6, 6, 6, 6, 6, 6, 6, 4)& 3& 28&0\\
13&(8, 6, 6, 6, 6, 6, 6, 6, 5)& -2& 15&2\\
13&(8, 6, 6, 6, 6, 6, 6, 6, 6)& -8& 5&13\\
13&*(8, 7, 6, 5, 5, 5, 5, 5, 5)& 21& 91& 1\\
13&*(8, 7, 7, 4, 4, 4, 4, 4, 4)& 44& 154& 2\\
13&*(8, 8, 5, 5, 5, 5, 5, 5, 5)& 19& 79& 4 \\
13&*(8, 8, 6, 4, 4, 4, 4, 4, 4)& 43& 148& 4\\
\hline

12&*(8, 6, 5, 5, 5, 5, 5, 5, 5)& 7& 35&1 \\
12&*(8, 6, 6, 4, 4, 4, 4, 4, 4)& 31& 105& 2\\
12&*(8, 7, 5, 4, 4, 4, 4, 4, 4)& 30& 100& 4\\
12&*(8, 8, 4, 4, 4, 4, 4, 4, 4)& 27& 85&  10\\
\hline

11&*(8, 5, 5, 4, 4, 4, 4, 4, 4)& 18& 56& 2\\
11&*(8, 6, 4, 4, 4, 4, 4, 4, 4)& 17& 52& 4\\
11&*(8, 6, 5, 3, 3, 3, 3, 3, 3)& 36& 98& 5\\
11&*(8, 7, 4, 3, 3, 3, 3, 3, 3)& 34& 90& 10\\
11&*(8, 8, 3, 3, 3, 3, 3, 3, 3)& 30& 74& 20\\
\hline
10&*(8, 4, 4, 4, 4, 4, 4, 4, 4)& 5& 14&8\\
10&*(8, 5, 4, 3, 3, 3, 3, 3, 3)& 24& 56&5\\
10&*(8, 6, 3, 3, 3, 3, 3, 3, 3)& 22& 50&10\\
\hline
9&*(8, 4, 3, 3, 3, 3, 3, 3, 3)& 12& 21&11\\
\hline
\end{tabular}
\end{table}

\clearpage

\begin{table}[!ht]
\label{degree-2m}
\centering
\caption{The case  $d=2m$}
\begin{tabular}{|c|c|c|c|c|}
\hline
$(m_1,m_2,m_3,m_4,m_5,m_6,m_7,m_8,m_9)$& $q$ & $\hh^0$ & $\hh^1$ \\
\hline
(6, 6, 6, 6, 6, 5, 5, 5, 5)& 4& 35&0 \\
(6, 6, 6, 6, 6, 6, 5, 5, 4)& 3& 29&0\\
(6, 6, 6, 6, 6, 6, 5, 5, 5)& -2& 16&2\\
(6, 6, 6, 6, 6, 6, 6, 4, 4)& 2& 23&0\\
(6, 6, 6, 6, 6, 6, 6, 5, 3)& 1& 18&0\\
(6, 6, 6, 6, 6, 6, 6, 5, 4)& -3&11&3\\
(6, 6, 6, 6, 6, 6, 6, 5, 5)& -8&6&13\\
\hline
(7, 7, 7, 7, 7, 6, 6, 6, 6)& 1& 36&0\\
(7, 7, 7, 7, 7, 7, 6, 6, 5)& 0& 29&0\\
(7, 7, 7, 7, 7, 7, 6, 6, 6)& -6& 16&8\\
(7, 7, 7, 7, 7, 7, 7, 5, 4)& 4& 37&0\\
(7, 7, 7, 7, 7, 7, 7, 5, 5)& -1& 23&1\\
(7, 7, 7, 7, 7, 7, 7, 6, 3)& 2& 26&0\\
(7, 7, 7, 7, 7, 7, 7, 6, 4)& -2& 18&2\\
(7, 7, 7, 7, 7, 7, 7, 6, 5)& -7& 1&10\\
(7, 7, 7, 7, 7, 7, 7, 6, 6)& -13&6&26\\
\hline
(8, 8, 8, 8, 7, 7, 7, 7, 7)& 5& 69&0\\
(8, 8, 8, 8, 8, 7, 7, 7, 6)& 4& 61&0\\
(8, 8, 8, 8, 8, 7, 7, 7, 7)& -3& 36& 3\\
(8, 8, 8, 8, 8, 8, 7, 6, 6)& 3& 53& 0\\
(8, 8, 8, 8, 8, 8, 7, 7, 5)& 2& 46&0\\
(8, 8, 8, 8, 8, 8, 7, 7, 6)& -4& 29&4\\
(8, 8, 8, 8, 8, 8, 7, 7, 7)& -11& 16&19\\
(8, 8, 8, 8, 8, 8, 8, 6, 5)& 1& 38&0\\
(8, 8, 8, 8, 8, 8, 8, 6, 6)& -5& 23&6\\
(8, 8, 8, 8, 8, 8, 8, 7, 3)& 3& 35&0\\
(8, 8, 8, 8, 8, 8, 8, 7, 4)& -1& 26&1\\
(8, 8, 8, 8, 8, 8, 8, 7, 5)& -6& 18&8\\
(8, 8, 8, 8, 8, 8, 8, 7, 6)& -12& 11&22\\
(8, 8, 8, 8, 8, 8, 8, 7, 7)& -19& 6&45\\
\hline
\end{tabular}
\end{table}

\end{document}